\documentclass{amsart}
\usepackage{amssymb}
\usepackage[colorlinks=true,linkcolor=blue,citecolor=blue,urlcolor=blue]{hyperref}

\newtheorem{theorem}{Theorem}[section]
\newtheorem{lemma}[theorem]{Lemma}
\newtheorem{proposition}[theorem]{Proposition}

\newtheorem{corollary}[theorem]{Corollary}
\newtheorem{definition}[theorem]{Definition}

\numberwithin{equation}{section}
\allowdisplaybreaks[4]

\begin{document}

\title[$B_2^{(1)}$ character recurrences and formulas]
{Character recurrences and level one formulas for $B_2^{(1)}$}

\author{Miroslav Jerkovi\'c}

\address{University of Zagreb Faculty of Chemical Engineering and Technology, Maruli\'cev trg 19, HR-10000 Zagreb, Croatia}
\email{mjerkov@fkit.unizg.hr}

\date{}

\begin{abstract}
    We obtain recurrence relations for formal characters of Feigin--Stoyanovsky
    type subspaces of standard $B_2^{(1)}$-modules at fixed level. The recurrences
    are derived directly from the combinatorial bases described by difference and
    initial conditions. Unlike in type $A_\ell^{(1)}$, the right side in general
    involves several shifted characters. We give explicit polynomial coefficients
    and prove that the resulting system, together with the constant term
    normalization, determines all formal characters at fixed level uniquely. At
    level $1$ we solve the recurrence system explicitly and obtain fermionic
    formulas for all three characters. We also relate these characters to the
    multigraded generating functions for exponent triples parametrizing the
    Meurman--Primc monomial basis of standard $A_1^{(1)}$-modules.
\end{abstract}

\maketitle
\enlargethispage{2pt}

\section{Introduction}\label{S:intro}

The study of combinatorial bases of standard modules for affine Lie algebras
originates in the work of Lepowsky and Wilson on Rogers--Ramanujan type
identities (cf.\ \cite{LW1,LW2}). By using Lepowsky--Wilson's approach with
untwisted vertex operators, Lepowsky and Primc obtained in \cite{LP} a
combinatorial basis of standard $A_1^{(1)}$-modules in the homogeneous picture.

Later, Feigin and Stoyanovsky introduced a class of distinguished subspaces of
standard modules \cite{FS}. Inspired by this construction, Feigin--Stoyanovsky
type subspaces became a useful setting for constructing combinatorial bases
related to standard modules of affine Lie algebras (cf.\ \cite{T,P,BPT}). In
many cases, these subspaces admit combinatorial bases described in terms of
colored partitions satisfying certain difference and initial conditions. Such
bases were obtained for affine Lie algebras of type $A_\ell^{(1)}$ in \cite{T},
for type $C_\ell^{(1)}$ in \cite{BPT}, and for type $B_2^{(1)}$ in \cite{P}.

Formal characters of Feigin--Stoyanovsky type subspaces are generating
functions of the colored partitions that parametrize the bases. Relations among
such generating functions at fixed level give recurrence relations for the
characters. For affine Lie algebras of type $A_\ell^{(1)}$, a system of such
recurrences was obtained in \cite{J} from exact sequences constructed by means
of intertwining operators and a simple current operator.

In this paper we derive analogous recurrence relations at fixed level for the
affine Lie algebra $B_2^{(1)}$ directly from the combinatorial basis of
\cite{P}. In contrast with the $A_\ell^{(1)}$ case, the right side of the
recurrence in Theorem~\ref{T:main} in general involves several shifted
characters. This reflects the fact that, after lowering the relevant initial
conditions, several degree $-1$ triples have to be removed. Together with the
relation for the case $k_0+k_2=0$ and the constant term normalization, these
recurrences determine all formal characters at fixed level uniquely.

At level $1$, the recurrence system can be solved explicitly.
Theorem~\ref{T:level-one-character-formulas} gives fermionic formulas for the
three characters, written in terms of a single basic series and its shifts.
Finally, we compare the inequalities for the $B_2^{(1)}$ basis of \cite{P} with
those for the exponent triples parametrizing the Meurman--Primc monomial basis
of standard $A_1^{(1)}$-modules \cite{MP}. This gives a decomposition of the
multigraded generating function of these exponent triples into formal
characters.

The paper is organized as follows. Section~\ref{S:setting} fixes the notation,
recalls the basis theorem of \cite{P}, and introduces the formal characters.
Section~\ref{S:boundary} describes the difference $\Pi_K\setminus\Pi_{K^{-}}$
and establishes the deletion bijection. Section~\ref{S:recurrence} derives the
recurrence relations and proves uniqueness of the resulting system.
Section~\ref{S:level-one} solves the level $1$ system explicitly.
Section~\ref{S:meurman-primc} relates the exponent triples parametrizing the
Meurman--Primc monomial basis of standard $A_1^{(1)}$-modules \cite{MP} to the
formal characters introduced in Section~\ref{S:setting}. It then uses the level
$1$ formulas of Section~\ref{S:level-one} to write the level $1$ Meurman--Primc
generating functions explicitly.

The author thanks Mirko Primc for his valuable suggestions.

\section{Preliminaries}\label{S:setting}

\subsection{Affine Lie algebra and standard modules}

Let $\mathfrak g$ be a simple Lie algebra of type $B_2$ and let $\mathfrak h$
be a Cartan subalgebra of $\mathfrak g$. We realize the root system in $\mathbb
    R^2$, with canonical basis $\varepsilon_1,\varepsilon_2$, as
$R=\{\pm(\varepsilon_1-\varepsilon_2),\pm(\varepsilon_1+\varepsilon_2)\}
    \cup\{\pm\varepsilon_1,\pm\varepsilon_2\}$. We fix simple roots
$\alpha_1=\varepsilon_1-\varepsilon_2$ and $\alpha_2=\varepsilon_2$. The
maximal root is $\theta=\varepsilon_1+\varepsilon_2$.

For the coweight $\omega=\omega_1=\varepsilon_1$ we have $\Gamma=\{\alpha\in
    R\mid \alpha(\omega)=1\}
    =\{\varepsilon_1-\varepsilon_2,\varepsilon_1,\varepsilon_1+\varepsilon_2\}$.
This gives the $\mathbb Z$-grading $\mathfrak g=\mathfrak g_{-1}\oplus\mathfrak
    g_0\oplus\mathfrak g_1$, where
\[
    \mathfrak g_{\pm 1}
    =
    \sum_{\alpha\in\pm\Gamma}\mathfrak g_\alpha,
    \qquad
    \mathfrak g_0
    =
    \mathfrak h+\sum_{\alpha(\omega)=0}\mathfrak g_\alpha .
\]
We write $\underline{2}=\varepsilon_1-\varepsilon_2$, $0=\varepsilon_1$, and
$2=\varepsilon_1+\varepsilon_2$, so that $\Gamma=\{\underline{2},0,2\}$, and
denote the corresponding root vectors by $x_{\underline{2}},x_0,x_2$. The
elements of $\Gamma$ are the colors. We use the same symbols
$\underline{2},0,2$ for these colors.

Let $\tilde{\mathfrak g}=\mathfrak g\otimes\mathbb C[t,t^{-1}] \oplus\mathbb C
    c\oplus\mathbb C d$ be the affine Lie algebra of type $B_2^{(1)}$, and write
$x(n)=x\otimes t^n$ for $x\in\mathfrak g$ and $n\in\mathbb Z$. The grading on
$\mathfrak g$ induces $\tilde{\mathfrak g}=\tilde{\mathfrak g}_{-1}\oplus
    \tilde{\mathfrak g}_0\oplus\tilde{\mathfrak g}_1$, where $\tilde{\mathfrak
        g}_1=\mathfrak g_1\otimes\mathbb C[t,t^{-1}]$ is a commutative Lie subalgebra
spanned by $x_{\underline{2}}(n),x_0(n),x_2(n)$, $n\in\mathbb Z$.

Let $L(\Lambda)$ be a standard $\tilde{\mathfrak g}$-module with dominant
integral highest weight $\Lambda=k_0\Lambda_0+k_1\Lambda_1+k_2\Lambda_2$ and
level $k=\Lambda(c)=k_0+k_1+k_2$, where $k_0,k_1,k_2\in\mathbb Z_{\geq 0}$. We
fix a highest weight vector $v_\Lambda\in L(\Lambda)$. We use $K=(k_0,k_1,k_2)$
to record the coefficients of $\Lambda$.

\subsection{Feigin--Stoyanovsky type subspaces and combinatorial bases}

For an integral dominant weight $\Lambda$, define the Feigin--Stoyanovsky type
subspace of $L(\Lambda)$ by $W(\Lambda)=U(\tilde{\mathfrak
        g}_1)v_\Lambda\subset L(\Lambda)$.

Let $\pi\colon\{x_\gamma(-j)\mid \gamma\in\Gamma,\ j\geq 1\}\to\mathbb Z_{\geq
            0}$ be a colored partition in these colors with finite support, and let
$x(\pi)=\prod x_\gamma(-j)^{\pi(x_\gamma(-j))}$ be the corresponding monomial
in $U(\tilde{\mathfrak g}_1)=S(\tilde{\mathfrak g}_1)$. We identify $\pi$ with
the sequence $a_1,b_1,c_1,a_2,b_2,c_2,\dots$, where $a_j=\pi(x_2(-j))$,
$b_j=\pi(x_0(-j))$, and $c_j=\pi(x_{\underline{2}}(-j))$. Then
\[
    x(\pi)=\cdots x_{\underline{2}}(-2)^{c_2}x_0(-2)^{b_2}x_2(-2)^{a_2}
    x_{\underline{2}}(-1)^{c_1}x_0(-1)^{b_1}x_2(-1)^{a_1}.
\]
For $K=(k_0,k_1,k_2)$, let $\Pi_K$ be the set of colored partitions satisfying
the difference conditions
\begin{equation}\tag{DC}\label{DC}
    \begin{aligned}
        c_{j+1}+b_{j+1}+c_j & \leq k, \\
        b_{j+1}+a_{j+1}+c_j & \leq k, \\
        a_{j+1}+c_j+b_j     & \leq k, \\
        a_{j+1}+b_j+a_j     & \leq k
    \end{aligned}
\end{equation}
for all $j\geq 1$, and the initial conditions
\begin{align}
    a_1     & \leq k_0, \tag{IC1}\label{IC1}     \\
    b_1+a_1 & \leq k_0+k_2, \tag{IC2}\label{IC2} \\
    c_1+b_1 & \leq k_0+k_2. \tag{IC3}\label{IC3}
\end{align}
The labels \hyperref[DC]{DC}, \hyperref[IC1]{IC1},
\hyperref[IC2]{IC2}, and \hyperref[IC3]{IC3} denote the difference
conditions and the three initial inequalities. The colored partitions
in $\Pi_K$ are the admissible partitions for $K$.
By Theorem~3.1 of \cite{P}, the vectors
\[
    x(\pi)v_\Lambda,\qquad \pi\in\Pi_K,
\]
form a basis of the level $k$ Feigin--Stoyanovsky type subspace $W(\Lambda)$.
We denote this basis by
\[
    \mathcal B_K=\{x(\pi)v_\Lambda\mid \pi\in\Pi_K\}.
\]
When no confusion can arise, $\mathcal B_K$ is identified with the set $\Pi_K$
of colored partitions parametrizing these basis vectors.

\subsection{Formal characters}

For $K=(k_0,k_1,k_2)$ and $\Lambda=k_0\Lambda_0+k_1\Lambda_1+k_2\Lambda_2$, we
define the formal character of $W(\Lambda)$ by
\[
    \chi_K(z_1,z_2,z_3,q)
    =\sum_{\pi\in\Pi_K}
    z_1^{|\pi|_2}z_2^{|\pi|_0}z_3^{|\pi|_{\underline{2}}}q^{\deg(\pi)},
\]
where
\[
    |\pi|_\gamma=\sum_{j\geq 1}\pi(x_\gamma(-j)),
    \qquad
    \deg(\pi)=\sum_{\gamma\in\Gamma}\sum_{j\geq 1}
    j\,\pi(x_\gamma(-j)).
\] We write $\mathbf z=(z_1,z_2,z_3)$, so that $z_1$, $z_2$, $z_3$ track the
colors $2$, $0$, $\underline{2}$, respectively. The character $\chi_K(\mathbf
    z,q)$ is the multigraded generating function of the colored partitions in
$\Pi_K$.

\section{\texorpdfstring{The set difference $\Pi_K\setminus\Pi_{K^{-}}$}{The set difference PiK minus PiKminus}}\label{S:boundary}

\subsection{Definition of \texorpdfstring{$K^{-}$}{K-minus}}

\begin{definition}\label{D:Kminus}
    For $K=(k_0,k_1,k_2)$ with $k_0+k_2>0$, define
    \[
        K^{-} = \begin{cases}
            (k_0,\,k_1+1,\,k_2-1) & \text{if } k_2>0,          \\
            (k_0-1,\,k_1+1,\,0)   & \text{if } k_2=0,\, k_0>0.
        \end{cases}
    \]
    In both cases $K^{-}$ is again at level $k$, and the initial conditions for
    $K^{-}$ have the same shape as the initial conditions for $K$, with the right
    sides of \hyperref[IC2]{IC2} and \hyperref[IC3]{IC3} reduced from $k_0+k_2$ to
    $k_0+k_2-1$. When $k_2=0$, the right side of \hyperref[IC1]{IC1} is also
    reduced from $k_0$ to $k_0-1$.
\end{definition}

Since the difference conditions are the same for $K$ and $K^{-}$, and the
initial conditions for $K^{-}$ are at least as restrictive as those for $K$, we
have $\Pi_{K^{-}}\subseteq\Pi_K$. Hence the difference of the two formal
characters is the generating function of $\Pi_K\setminus\Pi_{K^{-}}$:
\begin{equation}\label{E:LHS gen fun}
    \chi_K(\mathbf{z},q)-\chi_{K^{-}}(\mathbf{z},q)
    = \sum_{\pi\in\Pi_K\setminus\Pi_{K^{-}}}
    z_1^{|\pi|_2}\,z_2^{|\pi|_0}\,z_3^{|\pi|_{\underline{2}}}\,q^{\deg(\pi)}.
\end{equation}
We partition $\Pi_K\setminus\Pi_{K^{-}}$ by the degree $-1$ triple
$(a_1,b_1,c_1)$ of $\pi$ and analyze each part.

\subsection{The set \texorpdfstring{$\mathcal{S}(K)$}{S(K)}}

\begin{proposition}\label{P:S(K)}
    Assume $k_0+k_2>0$. The set of degree $-1$ triples
    $(a_1,b_1,c_1)$ of partitions in
    $\Pi_K\setminus\Pi_{K^{-}}$ is
    \[
        \mathcal{S}(K) = \left\{(a_1,b_1,c_1)\in\mathbb{Z}_{\geq 0}^3 \,\middle|\,
        \begin{gathered}
            a_1\leq k_0,\,
            b_1+a_1\leq k_0+k_2,\,
            c_1+b_1\leq k_0+k_2, \\
            b_1+a_1=k_0+k_2\ \textup{or}\ c_1+b_1=k_0+k_2
        \end{gathered}
        \right\}.
    \]
\end{proposition}
\begin{proof}
    The difference conditions are the same for $K$ and $K^{-}$. Hence a partition
    in $\Pi_K$ fails to lie in $\Pi_{K^{-}}$ precisely when it violates at least
    one initial condition for $K^{-}$. We analyze the two cases in the definition
    of $K^{-}$ separately.

    \textit{Case $k_2>0$:} The initial conditions for $K^{-}=(k_0,k_1+1,k_2-1)$ are
    \begin{align*}
        a_1     & \leq k_0,       \\
        b_1+a_1 & \leq k_0+k_2-1, \\
        c_1+b_1 & \leq k_0+k_2-1.
    \end{align*}
    Since \hyperref[IC1]{IC1} is unchanged, violation requires
    \[
        b_1+a_1\geq k_0+k_2
        \quad\text{or}\quad
        c_1+b_1\geq k_0+k_2.
    \]
    Together with the initial conditions for $K$, this implies $b_1+a_1=k_0+k_2$ or
    $c_1+b_1=k_0+k_2$.

    \textit{Case $k_2=0$:} The initial conditions for $K^{-}=(k_0-1,k_1+1,0)$ are
    \begin{align*}
        a_1     & \leq k_0-1, \\
        b_1+a_1 & \leq k_0-1, \\
        c_1+b_1 & \leq k_0-1.
    \end{align*}
    If $a_1\geq k_0$, then $b_1+a_1\geq k_0$. Hence violation is
    equivalent to
    \[
        b_1+a_1\geq k_0
        \quad\text{or}\quad
        c_1+b_1\geq k_0.
    \]
    Together with the initial conditions for $K$, this implies
    \[
        b_1+a_1=k_0
        \quad\text{or}\quad
        c_1+b_1=k_0.
    \]
    Since $k_2=0$, this is equivalent to $b_1+a_1=k_0+k_2$ or $c_1+b_1=k_0+k_2$.

    Conversely, every triple satisfying the defining inequalities of
    $\mathcal{S}(K)$ is the degree $-1$ triple of some partition in
    $\Pi_K\setminus\Pi_{K^{-}}$. Namely, set $a_j=b_j=c_j=0$ for all $j\geq2$. The
    resulting partition satisfies \eqref{DC} at every $j$, satisfies the initial
    conditions for $K$ by assumption, and violates an initial condition for $K^{-}$
    by the equality $b_1+a_1=k_0+k_2$ or $c_1+b_1=k_0+k_2$.
\end{proof}

\subsection{Constraints at degree \texorpdfstring{$-2$}{-2}}

\begin{proposition}\label{P:deg2 constraints}
    Assume $k_0+k_2>0$ and fix
    $(a_1,b_1,c_1)\in\mathcal{S}(K)$. For
    $\pi\in\Pi_K\setminus\Pi_{K^{-}}$ with degree $-1$ triple
    $(a_1,b_1,c_1)$, the conditions \eqref{DC} with $j=1$ are
    equivalent to
    \[
        \begin{aligned}
            a_2     & \leq k_1,   \\
            b_2+a_2 & \leq k-c_1, \\
            c_2+b_2 & \leq k-c_1.
        \end{aligned}
    \]
    These are exactly the initial conditions for $\Pi_{(k_1,c_1,k_0+k_2-c_1)}$.
\end{proposition}
\begin{proof}
    Evaluating the four inequalities in \eqref{DC} at $j=1$ gives
    \begin{equation}\label{E:deg2 constraints}
        \begin{aligned}
            a_2     & \leq \min(k-c_1-b_1,\,k-b_1-a_1), \\
            b_2+a_2 & \leq k-c_1,                       \\
            c_2+b_2 & \leq k-c_1.
        \end{aligned}
    \end{equation}
    The initial conditions for $\Pi_{(k_1,c_1,k_0+k_2-c_1)}$ are
    \begin{align*}
        a_2     & \leq k_1,   \\
        b_2+a_2 & \leq k-c_1, \\
        c_2+b_2 & \leq k-c_1.
    \end{align*}
    The last two inequalities already agree with
    \eqref{E:deg2 constraints}. For the first, the defining conditions
    of $\mathcal{S}(K)$ give $b_1+a_1\leq k_0+k_2$ and
    $c_1+b_1\leq k_0+k_2$ with at least one equality, so
    $\max(b_1+a_1,\,c_1+b_1)=k_0+k_2$. Hence
    \begin{align*}
        \min(k-c_1-b_1,\,k-b_1-a_1) & = k-\max(c_1+b_1,\,b_1+a_1) \\
                                    & = k-(k_0+k_2) = k_1.
    \end{align*}
\end{proof}

The values of $c_1$ on $\mathcal{S}(K)$ are exactly $0,\ldots,k_0+k_2$, since
$(0,k_0+k_2-i,i)\in\mathcal{S}(K)$ for each such $i$. The weights
$(k_1,c_1,k_0+k_2-c_1)$ in Proposition~\ref{P:deg2 constraints} therefore run
over $(k_1,i,k_0+k_2-i)$, $i=0,\ldots,k_0+k_2$.

\begin{lemma}\label{L:deletion bijection}
    Fix $(a_1,b_1,c_1)\in\mathcal{S}(K)$. For a partition
    $\pi=(a_j,b_j,c_j)_{j\geq1}$ with degree $-1$ triple
    $(a_1,b_1,c_1)$, define $\pi'=(a'_j,b'_j,c'_j)_{j\geq1}$ by
    \[
        \begin{aligned}
            a'_j & = a_{j+1}, \\
            b'_j & = b_{j+1}, \\
            c'_j & = c_{j+1}.
        \end{aligned}
    \]
    Then $\pi\mapsto\pi'$ is a bijection from the set of all
    $\pi\in\Pi_K\setminus\Pi_{K^{-}}$ with degree $-1$ triple $(a_1,b_1,c_1)$ onto
    $\Pi_{(k_1,c_1,k_0+k_2-c_1)}$.

    If
    \[
        d=\deg(\pi'),\qquad
        n_1=|\pi'|_2,\quad n_2=|\pi'|_0,\quad
        n_3=|\pi'|_{\underline{2}},
    \]
    then
    \[
        \deg(\pi)=(a_1+b_1+c_1)+d+n_1+n_2+n_3.
    \]
\end{lemma}
\begin{proof}
    We verify that the map $\pi\mapsto\pi'$ is well-defined, injective, and
    surjective, and then compute the degree.

    \textit{Well-defined:} By Proposition~\ref{P:S(K)}, the fixed triple
    $(a_1,b_1,c_1)$ satisfies the initial conditions for $K$ and fails the initial
    conditions for $K^{-}$. The initial conditions for
    $\Pi_{(k_1,c_1,k_0+k_2-c_1)}$ are
    \[
        \begin{aligned}
            a'_1      & \leq k_1,   \\
            b'_1+a'_1 & \leq k-c_1, \\
            c'_1+b'_1 & \leq k-c_1.
        \end{aligned}
    \]
    Since $a'_1=a_2$, $b'_1=b_2$, and $c'_1=c_2$, these are exactly the conditions
    obtained in Proposition~\ref{P:deg2 constraints}. For the conditions \eqref{DC}
    with $j\geq2$, reindexing by $i=j-1$ gives the difference conditions for $\pi'$
    with $i\geq1$.

    \textit{Injective:} The degree $-1$ triple $(a_1,b_1,c_1)$ is fixed, so
    distinct partitions with degree $-1$ triple $(a_1,b_1,c_1)$ have distinct
    images.

    \textit{Surjective:} Given $\pi'\in\Pi_{(k_1,c_1,k_0+k_2-c_1)}$, write
    $\pi'=(a'_j,b'_j,c'_j)_{j\geq 1}$. Let $\pi$ be the partition whose degree $-1$
    triple is $(a_1,b_1,c_1)$ and whose remaining entries are
    \[
        (a_j,b_j,c_j)=(a'_{j-1},b'_{j-1},c'_{j-1})\quad\text{for } j\geq 2.
    \]
    The initial conditions for $\pi'$ are exactly the inequalities obtained from
    \eqref{DC} at $j=1$ by the calculation in Proposition~\ref{P:deg2 constraints}.
    Therefore, the difference conditions for $\pi$ hold at $j=1$. For $j\geq2$, the
    difference conditions for $\pi$ follow from the difference conditions for
    $\pi'$ by reindexing. Because $(a_1,b_1,c_1)\in\mathcal{S}(K)$, the initial
    conditions for $K$ hold and the initial conditions for $K^{-}$ fail. Hence the
    constructed partition lies in $\Pi_K\setminus\Pi_{K^{-}}$ and has image $\pi'$.

    \textit{Degree statement:} The partition $\pi$ has degree
    \begin{align*}
        \sum_{j\geq 1}j(a_j+b_j+c_j)
         & = (a_1+b_1+c_1) + \sum_{j\geq 2}j(a'_{j-1}+b'_{j-1}+c'_{j-1})                      \\
         & = (a_1+b_1+c_1) + \sum_{i\geq 1}(i+1)(a'_i+b'_i+c'_i)                              \\
         & = (a_1+b_1+c_1) + \sum_{i\geq 1}i(a'_i+b'_i+c'_i) + \sum_{i\geq 1}(a'_i+b'_i+c'_i) \\
         & = (a_1+b_1+c_1) + d + (n_1+n_2+n_3).
    \end{align*}
\end{proof}

\section{The recurrence relations}\label{S:recurrence}

\begin{theorem}\label{T:main}
    Let $K=(k_0,k_1,k_2)$ be a level $k$ weight.

    If $k_0+k_2>0$, let $K^{-}$ and $\mathcal{S}(K)$ be as in
    Definition~\ref{D:Kminus} and Proposition~\ref{P:S(K)}. Then
    \begin{equation}\label{E:main}
        \chi_K(\mathbf{z},q)-\chi_{K^{-}}(\mathbf{z},q)
        = \sum_{i=0}^{k_0+k_2}\mathcal{P}_i(\mathbf{z},q)\,
        \chi_{(k_1,i,k_0+k_2-i)}(\mathbf{z}q,q),
    \end{equation}
    where $\mathbf{z}q=(z_1q,z_2q,z_3q)$ and
    \begin{equation}\label{E:Pm}
        \mathcal{P}_i(\mathbf{z},q) =
        \sum_{\substack{(a_1,b_1,c_1)\in\mathcal{S}(K)\\c_1=i}}
        z_1^{a_1}\,z_2^{b_1}\,z_3^{c_1}\,q^{a_1+b_1+c_1}.
    \end{equation}

    If $k_0+k_2=0$, then $K=(0,k,0)$ and
    \begin{equation}\label{E:base}
        \chi_{(0,k,0)}(\mathbf{z},q)
        =
        \chi_{(k,0,0)}(\mathbf{z}q,q).
    \end{equation}
\end{theorem}
\begin{proof}
    Assume first that $k_0+k_2>0$. By \eqref{E:LHS gen fun}, the left side of
    \eqref{E:main} is the generating function of $\Pi_K\setminus\Pi_{K^{-}}$. By
    Lemma~\ref{L:deletion bijection}, the set $\Pi_K\setminus\Pi_{K^{-}}$ is
    partitioned by the degree $-1$ triple $(a_1,b_1,c_1)\in\mathcal{S}(K)$, and for
    each fixed triple the deletion bijection maps the subset with degree $-1$
    triple $(a_1,b_1,c_1)$ onto $\Pi_{(k_1,c_1,k_0+k_2-c_1)}$.

    For a partition with degree $-1$ triple $(a_1,b_1,c_1)$ whose image $\pi'$
    satisfies $|\pi'|_2=n_1$, $|\pi'|_0=n_2$, $|\pi'|_{\underline{2}}=n_3$, and
    $\deg(\pi')=d$, the contribution to the generating function is
    \begin{align*}
        z_1^{a_1+n_1}\,z_2^{b_1+n_2}\,z_3^{c_1+n_3}\,
        q^{(a_1+b_1+c_1)+(d+n_1+n_2+n_3)}
         & = z_1^{a_1} z_2^{b_1} z_3^{c_1} q^{a_1+b_1+c_1}             \\
         & \quad\cdot z_1^{n_1} z_2^{n_2} z_3^{n_3} q^{d+n_1+n_2+n_3}.
    \end{align*}
    Summing over all images gives
    \[
        z_1^{a_1} z_2^{b_1} z_3^{c_1} q^{a_1+b_1+c_1}
        \cdot \chi_{(k_1,c_1,k_0+k_2-c_1)}(\mathbf{z}q,q).
    \]
    Summing over all $(a_1,b_1,c_1)\in\mathcal{S}(K)$ and grouping by $c_1=i$
    yields \eqref{E:main}.

    It remains to consider the case $k_0+k_2=0$. Then $K=(0,k,0)$, and the initial
    conditions
    \begin{align*}
        a_1     & \leq 0, \\
        b_1+a_1 & \leq 0, \\
        c_1+b_1 & \leq 0
    \end{align*}
    force $a_1=b_1=c_1=0$. Evaluating \eqref{DC} at $j=1$ gives
    \begin{align*}
        a_2     & \leq k, \\
        b_2+a_2 & \leq k, \\
        c_2+b_2 & \leq k,
    \end{align*}
    which are the initial conditions for $\Pi_{(k,0,0)}$. Thus deleting
    the zero degree $-1$ triple and reindexing gives a bijection from
    $\Pi_{(0,k,0)}$ to $\Pi_{(k,0,0)}$. After the zero degree $-1$
    triple is deleted and the remaining entries are reindexed, each
    remaining factor contributes one extra power of $q$, and hence
    \[
        \chi_{(0,k,0)}(\mathbf z,q)=\chi_{(k,0,0)}(\mathbf zq,q),
    \]
    which is \eqref{E:base}.
\end{proof}

\subsection{Explicit form of \texorpdfstring{$\mathcal{P}_i$}{Pi}}\label{SS:Pm explicit}

\begin{corollary}\label{C:Pm explicit}
    For $i=0,\ldots,k_0+k_2$,
    \begin{equation}\label{E:Pm explicit}
        \mathcal{P}_i(\mathbf{z},q) = z_3^i\bigg[
        q^{k_0+k_2+i}\!\sum_{a_1=i}^{k_0} z_2^{k_0+k_2-a_1}\,z_1^{a_1}
        \,+\, z_2^{k_0+k_2-i}q^{k_0+k_2}\!\sum_{a_1=0}^{\min(i-1,k_0)}(z_1q)^{a_1}
        \bigg],
    \end{equation}
    with the convention that the first sum is empty when $i>k_0$ and the second is empty when $i=0$.
\end{corollary}
\begin{proof}
    The triples in $\mathcal S(K)$ with $c_1=i$ split into two disjoint cases. If
    $b_1+a_1=k_0+k_2$, then $b_1=k_0+k_2-a_1$. Since $c_1+b_1\leq k_0+k_2$ and
    $b_1=k_0+k_2-a_1$, we have $c_1\leq a_1$, hence $i\leq a_1\leq k_0$. The
    contribution from triples satisfying $b_1+a_1=k_0+k_2$ to \eqref{E:Pm} is
    \[
        z_3^i q^{k_0+k_2+i}\sum_{a_1=i}^{k_0} z_2^{k_0+k_2-a_1}z_1^{a_1}.
    \]
    If $c_1+b_1=k_0+k_2$ and $b_1+a_1<k_0+k_2$, then $b_1=k_0+k_2-i$ and $0\leq
        a_1\leq \min(i-1,k_0)$. The contribution from triples satisfying
    $c_1+b_1=k_0+k_2$ and $b_1+a_1<k_0+k_2$ is
    \[
        z_3^i z_2^{k_0+k_2-i} q^{k_0+k_2} \sum_{a_1=0}^{\min(i-1,k_0)} (z_1q)^{a_1}.
    \]
    Summing the two contributions gives \eqref{E:Pm explicit}.
\end{proof}

\subsection{Uniqueness}\label{SS:uniqueness}

We regard the formal characters as elements of $\mathbb Z[z_1,z_2,z_3][[q]]$.
For each $d$ there are only finitely many colored partitions of degree $d$, and
therefore the coefficient of $q^d$ in every character is a polynomial in
$z_1,z_2,z_3$.

\begin{theorem}\label{T:uniqueness}
    The system \eqref{E:main}--\eqref{E:base}, together with the condition that the constant term of each series $\chi_K(\mathbf z,q)$ is equal to $1$, uniquely determines $\chi_K(\mathbf z,q)$ for all $K$ at level~$k$.
\end{theorem}
\begin{proof}
    Write
    \[
        \chi_K(\mathbf z,q)=\sum_{d\geq 0}f_K^{(d)}(\mathbf z)q^d.
    \]
    The condition on the constant terms gives $f_K^{(0)}(\mathbf z)=1$ for all
    weights $K$ at level~$k$. We prove by induction on $d$ that $f_K^{(d)}$ is
    uniquely determined for all such weights $K$.

    The recurrence \eqref{E:main} expresses $\chi_K-\chi_{K^{-}}$ in terms of
    products $\mathcal{P}_i(\mathbf{z},q)\cdot\chi_{K'}(\mathbf{z}q,q)$ for various
    $K'$. Since $k_0+k_2>0$, every triple in $\mathcal{S}(K)$ has
    $a_1+b_1+c_1\geq1$. Hence, by \eqref{E:Pm}, each $\mathcal{P}_i(\mathbf z,q)$
    has minimum $q$-degree at least $1$. Under the substitution
    $\mathbf{z}\mapsto\mathbf{z}q$, a monomial $\mathbf{z}^{\mathbf{n}}q^{d'}$ in
    $\chi_{K'}$ becomes $\mathbf{z}^{\mathbf{n}}q^{d'+|\mathbf{n}|}$, where
    $|\mathbf{n}|=n_1+n_2+n_3\geq 0$. Hence the coefficient of $q^d$ in the product
    $\mathcal{P}_i\cdot\chi_{K'}(\mathbf{z}q,q)$ depends only on coefficients
    $f_{K'}^{(d')}$ with $d'<d$.

    Fix $d\geq1$ and assume that all coefficients of degree $<d$ are known. We
    first determine $f_{(0,k,0)}^{(d)}$. By \eqref{E:base},
    \[
        \chi_{(0,k,0)}(\mathbf z,q)=\chi_{(k,0,0)}(\mathbf zq,q).
    \]
    The coefficient of $q^d$ on the right depends only on coefficients
    $f_{(k,0,0)}^{(d')}$ with $d'<d$: the empty monomial contributes only to degree
    $0$, and every non-empty monomial gains a positive power of $q$ under the
    substitution $\mathbf z\mapsto\mathbf zq$. Hence $f_{(0,k,0)}^{(d)}$ is
    determined.

    Now let $K$ be any level $k$ weight with $k_0+k_2>0$. Since each step $K\mapsto
        K^{-}$ decreases $k_0+k_2$ by one, iterating $K\mapsto K^{-}$ reaches $(0,k,0)$
    after finitely many steps. Along this chain, \eqref{E:main} determines
    \[
        f_K^{(d)}-f_{K^{-}}^{(d)}
    \]
    from coefficients of degree $<d$, which are already known. Starting from the
    already determined coefficient for $(0,k,0)$, we determine $f_K^{(d)}$
    successively along the chain. Thus all coefficients $f_K^{(d)}$ are uniquely
    determined.
\end{proof}

\section{Level one character formulas}\label{S:level-one}

Write
\begin{align*}
    \chi_{i,j,l}   & = \chi_{(i,j,l)}(\mathbf z,q),  \\
    \chi^+_{i,j,l} & = \chi_{(i,j,l)}(\mathbf zq,q).
\end{align*}

At level $1$, Theorem~\ref{T:main} and Corollary~\ref{C:Pm explicit} give
\begin{align}
    \chi_{1,0,0}-\chi_{0,1,0}
     & = q(z_2+z_1)\chi^+_{0,0,1}
    + qz_3(1+z_1q)\chi^+_{0,1,0},
    \label{E:lvl1 r1}
    \\[4pt]
    \chi_{0,0,1}-\chi_{0,1,0}
     & = qz_2\chi^+_{0,0,1}
    + qz_3\chi^+_{0,1,0},
    \label{E:lvl1 r2}
    \\[4pt]
    \chi_{0,1,0}
     & = \chi^+_{1,0,0}.
    \label{E:lvl1 r3}
\end{align}

Although Theorem~\ref{T:uniqueness} determines the characters recursively at
every level, the level $1$ system can be solved explicitly. For nonnegative
integers $m$, set
\[
    (q)_0=1,\qquad
    (q)_m=\prod_{j=1}^m(1-q^j).
\]
For integers $r$ and $s$ with $0\leq s\leq r$, define the Gaussian binomial
coefficient with base $q^2$ by
\[
    \binom{r}{s}_{q^2}
    =
    \prod_{j=1}^{s}
    \frac{1-q^{2(r-s+j)}}{1-q^{2j}}.
\]
We set $\binom{r}{s}_{q^2}=0$ if $s<0$ or $s>r$.

\begin{theorem}\label{T:level-one-character-formulas}
    Define
    \[
        F(\mathbf z,q)
        =
        \sum_{n_1,n_2,n_3\geq0}
        \frac{q^{n_1^2+n_3^2+\binom{n_2+1}{2}+n_1n_2+n_2n_3}}
        {(q)_{n_2}(q)_{n_1+n_3}}
        \binom{n_1+n_3}{n_1}_{q^2}\,
        z_1^{n_1}z_2^{n_2}z_3^{n_3}.
    \]
    Here $n_1,n_2,n_3$ record the multiplicities of the colors $2,0,\underline 2$,
    corresponding to the variables $z_1,z_2,z_3$, respectively. Then the level $1$
    characters are given by
    \begin{equation}\label{E:level-one-chi}
        \begin{aligned}
            \chi_{1,0,0}
             & = F(\mathbf z,q),  \\
            \chi_{0,1,0}
             & = F(\mathbf zq,q), \\
            \chi_{0,0,1}
             & =
            F\bigl((z_1q,\,z_2,\,z_3q),\,q\bigr)
            + qz_3\,F\bigl((z_1q^2,\,z_2q,\,z_3q^2),\,q\bigr).
        \end{aligned}
    \end{equation}
\end{theorem}

\begin{proof}
    By Theorem~\ref{T:uniqueness}, it is enough to show that the three series in
    \eqref{E:level-one-chi} satisfy \eqref{E:lvl1 r1}--\eqref{E:lvl1 r3} and have
    constant term $1$. The first two series have constant term $1$. For the third
    series, the first summand has constant term $1$, while the second summand is
    divisible by $q$. Relation \eqref{E:lvl1 r3} follows directly from
    \eqref{E:level-one-chi}.

    Let $f_{n_1,n_2,n_3}$ denote the coefficient of $z_1^{n_1}z_2^{n_2}z_3^{n_3}$
    in $F(\mathbf z,q)$ for $n_1,n_2,n_3\geq0$. We set $f_{n_1,n_2,n_3}=0$ whenever
    $n_1<0$, $n_2<0$, or $n_3<0$. Thus
    \[
        f_{n_1,n_2,n_3}
        =
        \frac{q^{n_1^2+n_3^2+\binom{n_2+1}{2}+n_1n_2+n_2n_3}}
        {(q)_{n_2}\,(q)_{n_1+n_3}}
        \binom{n_1+n_3}{n_1}_{q^2}.
    \]
    Comparing the powers of $q$ and the factors $(q)_{n_2}$ gives, for $0\leq m\leq
        n_2$,
    \begin{equation}\label{E:f-ratio}
        f_{n_1,n_2-m,n_3}
        =
        q^{\binom{m}{2}-m(n_1+n_2+n_3)}
        \frac{(q)_{n_2}}{(q)_{n_2-m}}\,
        f_{n_1,n_2,n_3}.
    \end{equation}

    \emph{Relation \eqref{E:lvl1 r2}.} First we prove
    \begin{equation}\label{E:level-one-chi-sum}
        \chi_{0,0,1}
        =
        \sum_{m\geq0}
        q^{\binom{m+1}{2}}z_2^m
        \left(
        F(\mathbf zq^{m+1},q)
        +q^{m+1}z_3\,F(\mathbf zq^{m+2},q)
        \right).
    \end{equation}
    The contribution of the terms containing $F(\mathbf zq^{m+1},q)$ to
    the coefficient of $z_1^{n_1}z_2^{n_2}z_3^{n_3}$ on the right side of
    \eqref{E:level-one-chi-sum} equals
    \begin{align*}
         &
        \sum_{m=0}^{n_2}
        q^{\binom{m+1}{2}+(m+1)(n_1+n_2-m+n_3)}
        f_{n_1,n_2-m,n_3}
        \\
         & =
        \sum_{m=0}^{n_2}
        q^{n_1+n_2+n_3-m}
        \frac{(q)_{n_2}}{(q)_{n_2-m}}\,
        f_{n_1,n_2,n_3}
        \\
         & =
        q^{n_1+n_3}\,(q)_{n_2}\,f_{n_1,n_2,n_3}
        \sum_{m=0}^{n_2}
        \frac{q^{n_2-m}}{(q)_{n_2-m}}
        \\
         & =
        q^{n_1+n_3}\,f_{n_1,n_2,n_3}.
    \end{align*}
    The second equality follows from \eqref{E:f-ratio}, and the last one
    from Euler's identity
    \[
        \sum_{m=0}^{n}\frac{q^m}{(q)_m}
        =
        \frac{1}{(q)_n}.
    \]
    This equals the coefficient of $z_1^{n_1}z_2^{n_2}z_3^{n_3}$ in
    \[
        F\bigl((z_1q,\,z_2,\,z_3q),\,q\bigr).
    \]

    For $n_3>0$, the contribution of the terms containing $q^{m+1}z_3F(\mathbf
        zq^{m+2},q)$ to the coefficient of $z_1^{n_1}z_2^{n_2}z_3^{n_3}$ on the right
    side of \eqref{E:level-one-chi-sum} equals
    \begin{align*}
         &
        \sum_{m=0}^{n_2}
        q^{
                \binom{m+1}{2}
                +m+1
                +(m+2)(n_1+n_2-m+n_3-1)
            }
        f_{n_1,n_2-m,n_3-1}
        \\
         & =
        \sum_{m=0}^{n_2}
        q^{2n_1+2n_2+2n_3-1-m}
        \frac{(q)_{n_2}}{(q)_{n_2-m}}\,
        f_{n_1,n_2,n_3-1}
        \\
         & =
        q^{2n_1+n_2+2n_3-1}
        (q)_{n_2}\,f_{n_1,n_2,n_3-1}
        \sum_{m=0}^{n_2}
        \frac{q^{n_2-m}}{(q)_{n_2-m}}
        \\
         & =
        q^{2n_1+n_2+2n_3-1}
        f_{n_1,n_2,n_3-1}.
    \end{align*}
    For $n_3=0$ the contribution is zero. Therefore the two contributions
    give the coefficients of the two summands in
    \[
        F\bigl((z_1q,\,z_2,\,z_3q),\,q\bigr)
        +
        qz_3\,F\bigl((z_1q^2,\,z_2q,\,z_3q^2),\,q\bigr),
    \]
    and \eqref{E:level-one-chi-sum} follows from \eqref{E:level-one-chi}.

    Substituting $\mathbf z\mapsto\mathbf zq$ in \eqref{E:level-one-chi-sum} and
    multiplying by $qz_2$ gives
    \begin{align*}
        qz_2\,\chi^+_{0,0,1}
         & =
        \sum_{m\geq0}
        q^{\binom{m+1}{2}+m+1}
        z_2^{m+1}
        \left(
        F(\mathbf zq^{m+2},q)
        +q^{m+2}z_3\,F(\mathbf zq^{m+3},q)
        \right)
        \\
         & =
        \sum_{m\geq1}
        q^{\binom{m+1}{2}}
        z_2^m
        \left(
        F(\mathbf zq^{m+1},q)
        +q^{m+1}z_3\,F(\mathbf zq^{m+2},q)
        \right).
    \end{align*}
    This is the part of \eqref{E:level-one-chi-sum} with $m\geq1$. The
    summand with $m=0$ is
    \[
        F(\mathbf zq,q)+qz_3\,F(\mathbf zq^2,q)
        =
        \chi_{0,1,0}+qz_3\,\chi^+_{0,1,0}.
    \]
    Together these two parts give
    \[
        \chi_{0,0,1}
        =
        \chi_{0,1,0}
        +qz_3\,\chi^+_{0,1,0}
        +qz_2\,\chi^+_{0,0,1},
    \]
    which is \eqref{E:lvl1 r2}.

    \emph{Relation \eqref{E:lvl1 r1}.} The coefficient of
    $z_1^{n_1}z_2^{n_2}z_3^{n_3}$ in $\chi_{1,0,0}-\chi_{0,1,0}$ is
    \[
        (1-q^{n_1+n_2+n_3})\,f_{n_1,n_2,n_3}.
    \]
    The right side of \eqref{E:lvl1 r1} is
    \[
        qz_2\,\chi^+_{0,0,1}
        +qz_1\,\chi^+_{0,0,1}
        +qz_3\,\chi^+_{0,1,0}
        +q^2z_1z_3\,\chi^+_{0,1,0}.
    \]
    Using \eqref{E:level-one-chi}, the coefficients of
    $z_1^{n_1}z_2^{n_2}z_3^{n_3}$ in these four terms are
    \begin{align}
        qz_2\,\chi^+_{0,0,1}
         & :\
        q^{2n_1+n_2+2n_3}\,f_{n_1,n_2-1,n_3}
        \notag
        \\
         & \quad
        +q^{3n_1+2n_2+3n_3-2}\,
        f_{n_1,n_2-1,n_3-1},
        \label{E:lvl1 coeff z2}
        \\[2pt]
        qz_1\,\chi^+_{0,0,1}
         & :\
        q^{2n_1+n_2+2n_3-1}\,f_{n_1-1,n_2,n_3}
        \notag
        \\
         & \quad
        +q^{3n_1+2n_2+3n_3-3}\,
        f_{n_1-1,n_2,n_3-1},
        \label{E:lvl1 coeff z1}
        \\[2pt]
        qz_3\,\chi^+_{0,1,0}
         & :\
        q^{2n_1+2n_2+2n_3-1}\,
        f_{n_1,n_2,n_3-1},
        \label{E:lvl1 coeff z3}
        \\[2pt]
        q^2z_1z_3\,\chi^+_{0,1,0}
         & :\
        q^{2n_1+2n_2+2n_3-2}\,
        f_{n_1-1,n_2,n_3-1}.
        \label{E:lvl1 coeff z1z3}
    \end{align}

    The definition of the Gaussian binomial coefficients gives the following
    identities. For $n_1>0$,
    \begin{equation}\label{E:gb1}
        \frac{(q)_{n_1+n_3}}{(q)_{n_1+n_3-1}}\,
        \frac{\binom{n_1+n_3-1}{n_1-1}_{q^2}}
        {\binom{n_1+n_3}{n_1}_{q^2}}
        =
        \frac{1-q^{2n_1}}{1+q^{n_1+n_3}}.
    \end{equation}
    For $n_3>0$,
    \begin{equation}\label{E:gb2}
        \frac{(q)_{n_1+n_3}}{(q)_{n_1+n_3-1}}\,
        \frac{\binom{n_1+n_3-1}{n_1}_{q^2}}
        {\binom{n_1+n_3}{n_1}_{q^2}}
        =
        \frac{1-q^{2n_3}}{1+q^{n_1+n_3}}.
    \end{equation}
    For $n_1,n_3>0$,
    \begin{equation}\label{E:gb3}
        (1+q^{n_1+n_3-1})\,
        \frac{(q)_{n_1+n_3}}{(q)_{n_1+n_3-2}}\,
        \frac{\binom{n_1+n_3-2}{n_1-1}_{q^2}}
        {\binom{n_1+n_3}{n_1}_{q^2}}
        =
        \frac{(1-q^{2n_1})(1-q^{2n_3})}
        {1+q^{n_1+n_3}}.
    \end{equation}

    Dividing the sum of \eqref{E:lvl1 coeff z2}--\eqref{E:lvl1 coeff z1z3} by
    $f_{n_1,n_2,n_3}$ and using \eqref{E:f-ratio}, \eqref{E:gb1}, \eqref{E:gb2} and
    \eqref{E:gb3} gives
    \begin{align*}
         &
        q^{n_1+n_3}(1-q^{n_2})
        \\
         & \quad
        +\frac{
        q^{2n_3}(1-q^{2n_1})
        +q^{2n_1}(1-q^{n_2})(1-q^{2n_3})
        }{
        1+q^{n_1+n_3}
        }
        \\
         & \quad
        +\frac{
        q^{2n_1+n_2}(1-q^{2n_3})
        +(1-q^{2n_1})(1-q^{2n_3})
        }{
        1+q^{n_1+n_3}
        }.
    \end{align*}
    The numerator of the two fractions equals
    \begin{align*}
         &
        q^{2n_3}(1-q^{2n_1})
        +q^{2n_1}(1-q^{n_2})(1-q^{2n_3})
        \\
         & \quad
        +q^{2n_1+n_2}(1-q^{2n_3})
        +(1-q^{2n_1})(1-q^{2n_3})
        \\
         & =
        q^{2n_3}(1-q^{2n_1})
        +q^{2n_1}(1-q^{2n_3})
        +(1-q^{2n_1})(1-q^{2n_3})
        \\
         & =
        1-q^{2n_1+2n_3}.
    \end{align*}
    The total contribution divided by $f_{n_1,n_2,n_3}$ is therefore
    \begin{align*}
         &
        q^{n_1+n_3}(1-q^{n_2})
        +
        \frac{1-q^{2n_1+2n_3}}{1+q^{n_1+n_3}}
        \\
         & =
        q^{n_1+n_3}(1-q^{n_2})
        +1-q^{n_1+n_3}
        \\
         & =
        1-q^{n_1+n_2+n_3}.
    \end{align*}
    Multiplying by $f_{n_1,n_2,n_3}$ gives the coefficient of
    $z_1^{n_1}z_2^{n_2}z_3^{n_3}$ in
    $\chi_{1,0,0}-\chi_{0,1,0}$. This proves \eqref{E:lvl1 r1}.
\end{proof}

\section{Connection to the Meurman--Primc basis of standard \texorpdfstring{$A_1^{(1)}$}{A1(1)}-modules}\label{S:meurman-primc}

For $K=(k,0,0)$, it was observed in \cite{P} that the admissible partitions for
the $B_2^{(1)}$ basis of $W(k\Lambda_0)$ satisfy the same inequalities as the
exponent triples for the Meurman--Primc monomial basis of the standard
$A_1^{(1)}$-module $L(k\Lambda_0)$. The parametrization for all standard
$A_1^{(1)}$-modules is recalled in \cite{PT}. The comparison decomposes
according to the value of $c_0$. Combined with the recurrences of
Section~\ref{S:recurrence}, this decomposition expresses the multigraded
Meurman--Primc character of every standard $A_1^{(1)}$-module in terms of the
characters $\chi_K$, which are determined by Theorems~\ref{T:main}
and~\ref{T:uniqueness}. In this sense the $B_2^{(1)}$ recurrences at level $k$
also determine the corresponding $A_1^{(1)}$ characters.

We use the parametrization of the Meurman--Primc monomial basis recalled in
\cite{PT}. Fix $\bar{k}_0,\bar{k}_1\in\mathbb{Z}_{\geq 0}$ with
$k=\bar{k}_0+\bar{k}_1$, and let
$\widetilde{\mathcal{B}}_{(\bar{k}_0,\bar{k}_1)}$ denote the set of exponent
triples parametrizing this basis for the standard $A_1^{(1)}$-module of highest
weight $\bar{k}_0\Lambda_0+\bar{k}_1\Lambda_1$. It consists of triples
$(a_j,b_j,c_j)_{j\geq 0}$ of nonnegative integers satisfying
\[
    a_0=b_0=0,\qquad c_0\leq\bar{k}_1,\qquad a_1\leq\bar{k}_0,
\]
together with the difference conditions
\begin{equation}\label{E:MP difference conditions}
    \begin{aligned}
        a_j+b_j+a_{j+1}     & \leq k, \\
        c_j+b_j+a_{j+1}     & \leq k, \\
        c_j+b_{j+1}+a_{j+1} & \leq k, \\
        c_j+b_{j+1}+c_{j+1} & \leq k
    \end{aligned}
\end{equation}
for all $j\geq 0$.

We identify the three exponent sequences with the colors $2$, $0$,
$\underline{2}$ as follows:
\[
    a_j\leftrightarrow 2,\qquad
    b_j\leftrightarrow 0,\qquad
    c_j\leftrightarrow \underline{2}.
\]
After reordering the summands in each line, the inequalities \eqref{E:MP
    difference conditions} at $j\geq 1$ are exactly \eqref{DC}. For fixed
$c_0=m\in\{0,\ldots,\bar{k}_1\}$, the inequalities at $j=0$ become the initial
conditions for the $B_2^{(1)}$ weight $(\bar{k}_0,m,\bar{k}_1-m)$.

\begin{proposition}\label{P:c0 decomposition}
    For each $m\in\{0,\ldots,\bar{k}_1\}$, the map
    \[
        (a_j,b_j,c_j)_{j\geq 0}\longmapsto(a_j,b_j,c_j)_{j\geq 1}
    \]
    gives a bijection from the subset of
    $\widetilde{\mathcal{B}}_{(\bar{k}_0,\bar{k}_1)}$ consisting of exponent
    triples with $c_0=m$ onto $\Pi_{(\bar{k}_0,\,m,\,\bar{k}_1-m)}$.
\end{proposition}
\begin{proof}
    Fix $c_0=m$. The inequalities \eqref{E:MP difference conditions} at $j=0$ with
    $a_0=b_0=0$ give
    \begin{align*}
        a_1     & \leq k,   \\
        a_1     & \leq k-m, \\
        b_1+a_1 & \leq k-m, \\
        c_1+b_1 & \leq k-m.
    \end{align*}
    Since $m\leq\bar{k}_1$ implies $\bar{k}_0\leq k-m$, both
    $a_1\leq k$ and $a_1\leq k-m$ follow from $a_1\leq\bar{k}_0$. The
    inequality $a_1\leq\bar{k}_0$, together with
    \[
        \begin{aligned}
            b_1+a_1 & \leq k-m, \\
            c_1+b_1 & \leq k-m
        \end{aligned}
    \]
    is exactly the set of initial conditions for the $B_2^{(1)}$ weight
    $(\bar{k}_0,m,\bar{k}_1-m)$, because
    \[
        k-m=\bar{k}_0+\bar{k}_1-m=\bar{k}_0+(\bar{k}_1-m).
    \]
    The four inequalities in \eqref{DC} at $j\geq 1$ are the four inequalities
    in \eqref{E:MP difference conditions} at $j\geq 1$, read in reverse order.
    Hence the map sends the subset with $c_0=m$ into
    $\Pi_{(\bar{k}_0,m,\bar{k}_1-m)}$.

    Conversely, given $(a'_j,b'_j,c'_j)_{j\geq
                1}\in\Pi_{(\bar{k}_0,m,\bar{k}_1-m)}$, define $(a_j,b_j,c_j)_{j\geq0}$ by
    \[
        a_0=b_0=0,\qquad c_0=m,\qquad
        (a_j,b_j,c_j)=(a'_j,b'_j,c'_j)\quad (j\geq1).
    \]
    Then the resulting triple satisfies \eqref{E:MP difference conditions} at every
    $j\geq0$ together with the Meurman--Primc initial conditions, so it lies in
    $\widetilde{\mathcal{B}}_{(\bar{k}_0,\bar{k}_1)}$ with $c_0=m$.
\end{proof}

For $\pi=(a_j,b_j,c_j)_{j\geq 0}\in
    \widetilde{\mathcal{B}}_{(\bar{k}_0,\bar{k}_1)}$, define
\[
    |\pi|_{2}=\sum_{j\geq 0}a_j,\qquad
    |\pi|_{0}=\sum_{j\geq 0}b_j,\qquad
    |\pi|_{\underline{2}}=\sum_{j\geq 0}c_j
\]
and
\[
    \deg(\pi)=\sum_{j\geq 0}j(a_j+b_j+c_j).
\]
Define now the multigraded formal character of the standard $A_1^{(1)}$-module
$L(\bar{k}_0\Lambda_0+\bar{k}_1\Lambda_1)$, with respect to the Meurman--Primc
monomial basis, by
\[
    \widetilde{\chi}_{(\bar{k}_0,\bar{k}_1)}(\mathbf{z},q)
    =
    \sum_{\pi\in\widetilde{\mathcal{B}}_{(\bar{k}_0,\bar{k}_1)}}
    z_1^{|\pi|_2}\,z_2^{|\pi|_0}\,
    z_3^{|\pi|_{\underline{2}}}\,q^{\deg(\pi)}.
\]

\begin{corollary}\label{C:character identity}
    The character $\widetilde{\chi}_{(\bar{k}_0,\bar{k}_1)}$ satisfies
    \begin{equation}\label{E:character identity}
        \widetilde{\chi}_{(\bar{k}_0,\bar{k}_1)}(\mathbf{z},q)
        = \sum_{m=0}^{\bar{k}_1} z_3^{m}\,\chi_{(\bar{k}_0,\,m,\,\bar{k}_1-m)}(\mathbf{z},q).
    \end{equation}
\end{corollary}
\begin{proof}
    By Proposition~\ref{P:c0 decomposition}, the map sending
    $\pi=(a_j,b_j,c_j)_{j\geq 0}$ to $\pi'=(a_j,b_j,c_j)_{j\geq 1}$ is a bijection
    from the subset of $\widetilde{\mathcal{B}}_{(\bar{k}_0,\bar{k}_1)}$ with
    $c_0=m$ onto $\Pi_{(\bar{k}_0,m,\bar{k}_1-m)}$. Since $a_0=b_0=0$ and the terms
    with $j=0$ do not contribute to $\deg(\pi)$, deleting the $j=0$ entries
    preserves $|\pi|_2$, $|\pi|_0$, and $\deg(\pi)$. It changes only the
    $\underline{2}$-count:
    \[
        |\pi|_{\underline{2}}=c_0+|\pi'|_{\underline{2}}=m+|\pi'|_{\underline{2}}.
    \]
    Therefore the contribution of the triples with $c_0=m$ to
    $\widetilde{\chi}_{(\bar{k}_0,\bar{k}_1)}$ is
    \[
        z_3^{m}\,\chi_{(\bar{k}_0,m,\bar{k}_1-m)}(\mathbf{z},q).
    \]
    Summing over $m=0,\ldots,\bar{k}_1$ gives \eqref{E:character identity}.
\end{proof}

\begin{corollary}\label{C:level-one-MP}
    The multigraded Meurman--Primc characters of the two level $1$ standard
    $A_1^{(1)}$-modules are
    \begin{align*}
        \widetilde{\chi}_{(1,0)}(\mathbf z,q)
         & = F(\mathbf z,q),
        \\
        \widetilde{\chi}_{(0,1)}(\mathbf z,q)
         & = F\bigl((z_1q,\,z_2,\,z_3q),\,q\bigr)
        +qz_3F\bigl((z_1q^2,\,z_2q,\,z_3q^2),\,q\bigr)
        +z_3F(\mathbf zq,q),
    \end{align*}
    where $F$ is the series of Theorem~\ref{T:level-one-character-formulas}.
\end{corollary}
\begin{proof}
    Specialize Corollary~\ref{C:character identity} to $k=1$:
    $\widetilde{\chi}_{(1,0)}=\chi_{(1,0,0)}$ and
    $\widetilde{\chi}_{(0,1)}=\chi_{(0,0,1)}+z_3\chi_{(0,1,0)}$. Substituting the
    formulas of Theorem~\ref{T:level-one-character-formulas} gives the result.
\end{proof}

\end{document}